\documentclass[12pt,twoside,reqno]{amsart}
\usepackage{amsfonts}
\usepackage{amssymb}
\usepackage[notcite,notref]{}

\numberwithin{equation}{section}

\newtheorem{teo}{Theorem}[section]

\newtheorem{lem}[teo]{Lemma}
\newtheorem{cor}[teo]{Corollary}

\theoremstyle{definition}

\newtheorem{rem}[teo]{Remark}

\numberwithin{equation}{section}

\def\a{\alpha}
\def\b{\beta}
\def\l{\lambda }

\def\R{\mathbb{R}}
\def\N{\mathbb{N}}
\def\d{\delta}
\def\e{\varepsilon}

\def\f{\varphi}
\def\s{\sigma}

\newcommand{\ack}[1]{\footnote{#1}}

\begin{document}

\begin{center}
{\Large Sharp constants related to the triangle inequality in Lorentz spaces}
\end{center}
\medskip
\begin{center}
Sorina Barza$^1$, Viktor Kolyada\ack{An essential part of this
work was performed while the first and second named authors stayed
at the University of Barcelona as invited researchers. We express
our gratitude to the Department of Mathematics of the University
of Barcelona for the hospitality and excellent conditions.}, and
Javier Soria\ack{Research partially supported by grants
MTM2007-60500 and 2005SGR00556.}
\end{center}

\bigskip

\noindent
{\bf Abstract:} {\small We study  the Lorentz spaces
$L^{p,s}(R,\mu)$ in the range $1<p<s\le \infty$, for which the
standard functional
$$
||f||_{p,s}=\left(\int_0^\infty
(t^{1/p}f^*(t))^s\frac{dt}{t}\right)^{1/s}
  $$
 is only  a quasi-norm. We find the optimal constant in the triangle
inequality for this quasi-norm, which leads us  to consider the following decomposition norm:
 $$
  ||f||_{(p,s)}=\inf\bigg\{\sum_{k}||f_k||_{p,s}\bigg\},
 $$
    where the infimum is taken over all finite representations
 $f=\sum_{k}f_k. $
We also prove that the decomposition norm and the dual norm
  $$
   ||f||_{p,s}'= \sup\left\{ \int_R fg\,d\mu: ||g||_{p',s'}=1\right\}
 $$
 agree for all values $p,s>1$.}
\section{Introduction}

\noindent
The study of the normability of the Lorentz spaces $L^{p,s}(R,\mu)$ goes back to the work of G.G. Lorentz \cite{loran,lorpjm} (see also \cite{Sw, CS, CRS} for a more recent account of the normability results for the weighted Lorentz spaces). The condition defining these spaces is given in terms of the distribution function and, equivalently, the non-increasing rearrangement of $f$ (see  \cite{BS} for standard notations and basic definitions):
$$
\Vert f\Vert_{p,s}=\bigg(\int_0^{\infty}(t^{1/p}f^*(t))^s\,\frac{dt}t\bigg)^{1/s},
$$
with the usual modification if $s=\infty$. Lorentz proved that $\Vert\ \Vert_{p,s}$ is a norm, if and only if $1\le s\le p<\infty$, and the space $L^{p,s}(R,\mu)$ is always normable (i.e., there exists a norm  equivalent to $\Vert\ \Vert_{p,s}$), for the range  $1<p<s\le\infty$ (for the remaining cases it is known that $L^{p,s}(R,\mu)$ cannot be endowed with an equivalent norm). From now on we will only consider the range $1<p<\infty,  ~~1\le s\le \infty$.
\medskip

Note that the spaces $L^{p,s}$, with $p<s$, play an important role not only as
dual spaces for the Banach spaces $L^{p',s'}$ (see \cite{BS,Hu}). For example, 
they arise naturally in  limiting embeddings of Lipschitz spaces (\cite{ko}).
\medskip

The study of the normability for $p<s$ was carried out by means of the {\it maximal norm:}
$$
\Vert f\Vert_{p,s}^*=\bigg(\int_0^{\infty}(t^{1/p}f^{**}(t))^s\,\frac{dt}t\bigg)^{1/s},
$$
where
$$
f^{**}(t)=\frac1t\int_0^tf^*(x)\,dx.
$$
It is easy to see that $\Vert \ \Vert_{p,s}^*$ is always a norm.  Moreover, one can prove that $\Vert \ \Vert_{p,s}^*$ is
equivalent to $\Vert \ \Vert_{p,s}$, with the following optimal
estimates:
\begin{equation}\label{bcsm}
(p')^{1/s}\Vert f\Vert_{p,s}\le \Vert f\Vert_{p,s}^*\le p'\Vert
f\Vert_{p,s}
\end{equation}
(see \cite{SW,K}; as usual, $p'$ denotes the conjugate
exponent, $1/p+1/p'=1).$

\medskip

 As a consequence of the fact that $\Vert\
\Vert_{p,s}$ is equivalent to a norm, it is easy to see that it is
a quasi-norm satisfying the triangle inequality, uniformly on the
number of terms: there exists a constant $c_{p,s}>0$ such that,
for every finite collection $\{f_k\}_{k=1,\cdots,N}\subset
L^{p,s}(R,\mu)$:
\begin{equation}\label{trianunif}
\bigg\Vert\sum_{k=1}^Nf_j\bigg\Vert_{p,s}\le c_{p,s}\sum_{k=1}^N\Vert f_j\Vert_{p,s}.
\end{equation}
It can readily be proved the converse result; namely,  (\ref{trianunif}) is equivalent to the fact that $\Vert\ \Vert_{p,s}$ is normable and, even more, that an alternative equivalent norm is given by means of the following {\it decomposition norm:}
\begin{equation}\label{decomnorm}
\Vert f\Vert_{(p,s)}=\inf\bigg\{\sum_{k}||f_k||_{p,s}\bigg\},
\end{equation}
 where the infimum is taken over all finite representations $f=\sum_{k}f_k. $
\medskip

It is easy to prove that $\Vert \ \Vert_{(p,s)}$ is a norm,
equivalent to $\Vert\ \Vert_{p,s}$, that agrees with $\Vert \
\Vert_{p,s}$ if $1\le s\le p$. Moreover, the best constant in the
inequality $\Vert f\Vert_{p,s}\le c_{p,s}\Vert f\Vert_{(p,s)}$ is
the same as the optimal one in (\ref{trianunif}). One of the main
problems studied in this paper is to find the best constant in the
triangle inequality (\ref{trianunif}) and its continuous version, the
 Minkowski integral inequality (the control of these constants is sometimes very relevant
for estimating different type of
integral operators, where the use of the maximal norm and
the  inequalities (\ref{bcsm})
do not usually give optimal results).
\medskip

For the Lorentz norms we have the following version of H\"older's
inequality: if $f\in L^{p,s}(R,\mu)$ and $g\in L^{p',s'}(R,\mu)$
$(1<p<\infty,  ~~1\le s\le \infty)$, then
\begin{equation}\label{holder}
\left|\int_R fg\,d\mu\right|\le \Vert f\Vert_{p,s}||g||_{p',s'}
\end{equation}
(see \cite[p.\ 220]{BS}).
\medskip

In the theory of Banach Function Spaces ($L^{p,s}(R,\mu)$ is the canonical example in this context), and based on (\ref{holder}), it is also very natural to consider another norm defined in terms of the K\"othe duality, which is denoted as the {\it dual norm:}
\begin{equation}\label{dualnorm}
   ||f||_{p,s}'= \sup\left\{ \int_R fg\,d\mu: ||g||_{p',s'}=1\right\}.
\end{equation}
As in the case of the decomposition norm, $\Vert \ \Vert_{p,s}'$ is a norm, equivalent to $\Vert\ \Vert_{p,s}$ and $\Vert f\Vert_{p,s}'=\Vert f\Vert_{p,s}$, if $1\le s\le p$ (see (\ref{d4})). Therefore, $\Vert f\Vert_{p,s}'=\Vert f\Vert_{(p,s)}$ ($1\le s\le p$).
\medskip

The main result that we will prove in this paper shows that the
decomposition and dual norm agree in the whole range of indices
(Theorem~\ref{Main}), in spite of their quite different
definitions. We also find the best constants in the inequalities
relating either of these norms and $\Vert\ \Vert_{p,s}$ (see
(\ref{d3}), Theorem~\ref{D2}, and Remark~\ref{equds}). In
particular, these results give an alternative proof of the
normability of $L^{p,s}(R,\mu)$ with {\it optimal estimates}.  We would like to remark that,
while (\ref{bcsm}) follows easily from standard estimates, finding
the best constants in our context requires new ideas and much more complicated constructions.

\medskip

In Section 2 we prove several technical lemmas used in subsequent
sections. Section 3 introduces one of the key tools used in the
paper: the level function (see Theorems~\ref{Halperin1} and
\ref{D1}). Sections 4 and 5 are the core of the paper, dealing
with both the dual and decomposition norms, and proving the main
results already mentioned above. Finally, in Section~6 we obtain
the best constant in both the triangle and Minkowski's integral
inequalities for the Lorentz spaces.

\medskip

Throughout this paper $(R,\mu)$ denotes a $\s-$finite nonatomic
measure space.

\section{Auxiliary propositions}

\noindent
In this section we consider some auxiliary results that will be
used in the sequel. We begin with some general inequalities.

\begin{lem}\label{Cheb} Let $f$ and $g$ be non-increasing
nonnegative functions on $[0,1].$ Then
$$
\int_0^1f(x)dx\int_0^1g(x)dx\le \int_0^1f(x)g(x)dx.
$$
\end{lem}
This is the classical Chebyshev inequality (see, e.g., \cite{HM}).

\begin{cor} Let $g$ be a non-increasing
nonnegative function on $[0,1]$ and let $0<\a<1.$ Then
\begin{equation}\label{cheb}
\int_0^1g(x)dx\le (1-\a)\int_0^1g(x)x^{-\a}dx.
\end{equation}
\end{cor}

\begin{lem}\label{Ineq} Let $p,s\in (1,\infty).$ Then for any
$t\in [0,1]$
\begin{equation}\label{ineq}
(1-t^{s/p})^{1/s}(1-t^{s'/p'})^{1/s'}\le 1-t.
\end{equation}
\end{lem}
\begin{proof}

We will prove that for all $x,y \in (0,1)$
\begin{equation}\label{ineq1}
(1-x^s)^{1/s}(1-y^{s'})^{1/s'}\le 1-xy.
\end{equation}
Then (\ref{ineq}) will follow from (\ref{ineq1}) if we take
$x=t^{1/p}, y=t^{1/p'}$. To prove (\ref{ineq1}), fix $y$ and
denote
$$
\f(x)=1-xy-(1-x^s)^{1/s}(1-y^{s'})^{1/s'}.
$$
We have
$$
{\f}^{'}(x)=-y-\frac{x^{s-1}}{(1-x^s)^{1/s'}}(1-y^{s'})^{1/s'}.
$$
Set ${\f}^{'}(x)=0$. Then
$$
\frac{(1-x^s)^{1/s'}}{x^{s-1}}=\frac{(1-y^{s'})^{1/s'}}{y}
$$
and
$$
\left(\frac{1}{x^s}-1\right)^{1/s'}=\left(\frac{1}{y^{s'}}-1\right)^{1/s'}.
$$
This implies that $x^s=y^{s'}$, and hence, the function $\f$ has an  absolute
minimum for $x=y^{{1}/{(s-1)}}$ and this minimum is 0, which
proves (\ref{ineq1}).
\end{proof}\medskip

The following lemma   gives the sharp constant in the relation
between Lorentz norms with different second indices (see
\cite[p.\ 192]{SW}).
\begin{lem}\label{Lorentz} Let $1\le p<\infty$ and $1\le r<s\le
\infty.$ Then, for any function $f\in L^{p,r}(R,\mu)$
\begin{equation}\label{lorentz}
\left(\frac{p}{s}\right)^{1/s}\|f\|_{p,s}\le
\left(\frac{p}{r}\right)^{1/r}\|f\|_{p,r}.
\end{equation}
\end{lem}

 We consider now some auxiliary statements related to dual norm
 and decomposition norm.
\begin{lem}\label{dual} Let
$f\in L^{p,s}(R,\mu)~~~(1< p <\infty, 1\le s\le \infty).$ Then
\begin{equation}\label{2.1}
||f||_{p,s}'= ||f^*||_{p,s}'.
\end{equation}
\end{lem}

The proof can be found in \cite[p.\ 45-49]{BS}.

\begin{lem}\label{Dual-dec}  Let $f\in L^{p,s}(R,\mu)~~~(1<p<\infty, ~1\le s\le \infty).$
Then
\begin{equation}\label{dual-dec}
||f||'_{p,s}\le ||f||_{(p,s)}.
\end{equation}
\begin{proof}  Let $g\in
L^{p',s'}(R,\mu)$ and let
\begin{equation}\label{dual-dec1}
f=\sum_k f_k.
\end{equation}
Then, by H\"older's inequality (\ref{holder}),
$$
\int_R|fg|d\mu\le \sum_k\int_R|f_kg| d\mu\le
||g||_{p',s'}\sum_k||f_k||_{p,s}.
$$
Taking infimum over all representations (\ref{dual-dec1}), we
obtain (\ref{dual-dec}).
\end{proof}\medskip

\end{lem}

We shall use the following properties of the decomposition norm.

\begin{lem}\label{Decomp} Let $f\in L^{p,s}(R,\mu)$ $(1< p<\infty, ~~1\le s\le \infty)$. Then:
\begin{enumerate}
\item the equality
\begin{equation}\label{decomp1}
||f||_{(p,s)}=\inf\bigg\{\sum_{k}||f_k||_{p,s}\bigg\},
\end{equation}
holds, where the infimum is taken over all finite sequences $\{f_k\}$
such that $f_k\ge 0$ and
$$
|f(x)|=\sum_{k}f_k(x);
$$

\item if $0\le g\le f,$ then $||g||_{(p,s)}\le ||f||_{(p,s)};$

\item if $ 0\le g_n\le f$  and $g_n(x)\uparrow f(x)$ $\mu-$almost
everywhere on $R$, then $||g_n||_{(p,s)}\to ||f||_{(p,s)}.$

\end{enumerate}

\end{lem}
\begin{proof} Denote by $\s$ the right hand side of (\ref{decomp1}). We have
$f=\sum_k g_k$, where $g_k=f_k \operatorname{sign} f$ and
therefore
$$
\|f\|_{(p,s)}\le \sum_k \|g_k\|_{p,s}=\sum_k \|f_k\|_{p,s}.
$$
Thus, $\|f\|_{(p,s)}\le \s$. On the other hand, for any $\e>0$
there exists a representation $f=\sum_k g_k$ such that
$$
\|f\|_{(p,s)}> \sum_k \|g_k\|_{p,s}-\e.
$$
We have $|f|\le \sum_k |g_k|\equiv G$. Set $f_k=|fg_k|/G$. Then
$\|f_k\|_{p,s}\le \|g_k\|_{p,s}$ and $|f|=\sum_k f_k$. Thus
$\|f\|_{(p,s)}\ge \s-\e$, which proves (\ref{decomp1}). Further,
statement (2) follows immediately from statement (1). To prove
(3), observe that $||f-g_n||_{p,s}\to 0$ (see \cite[p. 41]{BS}).
Since
$$
||g_n||_{(p,s)}\le ||f||_{(p,s)}\le ||g_n||_{(p,s)} +
||f-g_n||_{p,s},
$$
we obtain (3).
\end{proof}\medskip

\begin{lem}\label{Decomp1}
 For each  $f\in L^{p,s}(R,\mu)$
\begin{equation}\label{decomp2}
||f||_{(p,s)}\le ||f^*||_{(p,s)}.
\end{equation}
\end{lem}
\begin{proof} It is known that there exists a measure preserving transformation
$\sigma:R\rightarrow (0,\mu(R))$ such that
\begin{equation}
f(x)=f^{*}(\sigma(x)),\quad\mu \text{-a.e. on}\quad R,
\end{equation}
(see \cite[p.\ 82, 83]{BS}).
 Let $f^*(t)=\sum_{k=1}^N g_k(t),~~ g_k\ge 0$. Then
$f(x)=\sum_{k=1}^N g_k(\sigma(x))$. Since $g_k\circ \sigma$ and
$g_k$ are equimeasurable, we have that
$$
\|f\|_{(p,s)}\le \sum_{k=1}^N\|g_k\circ\sigma\|_{p,s}=\sum_{k=1}^N
\|g_k\|_{p,s}.
$$
This implies (\ref{decomp2}).
\end{proof}\medskip

It will be proved below that for any $f$ we have   the equality in
(\ref{decomp2}).

\vskip 6pt
\begin{lem}\label{limiting} Let $1<p<\infty.$ Assume that $f\in L^{p,s_0}(R,\mu)$
for some $p\le s_0<\infty.$ Then
\begin{equation}\label{limiting1}
||f||_{p,\infty}=\lim_{s\to\infty}||f||_{p,s}
\end{equation}
and
\begin{equation}\label{limiting2}
||f||_{(p,\infty)}=\lim_{s\to\infty}||f||_{(p,s)}.
\end{equation}
\end{lem}
\begin{proof} To prove (\ref{limiting1}), we can assume that
$\mu(\operatorname{supp}f)<\infty.$ Then (\ref{limiting1}) follows
from a similar property for the  $L^s$-norm (see \cite[p.\ 226]{DiB}).

We shall prove (\ref{limiting2}). By Lemma \ref{Decomp}, we can
assume that $f\ge 0$ and consider only  representations
\begin{equation}\label{limiting3}
f=\sum_{k=1}^N f_k,\quad\text{where}\quad f_k\ge 0.
\end{equation}
For an arbitrary representation (\ref{limiting3}) we have that, for any
$s>s_0$
$$
||f||_{(p,s)}\le \sum_{k=1}^N ||f_k||_{p,s}.
$$
By (\ref{limiting1}), we obtain that
$$
\varlimsup_{s\to\infty}||f||_{(p,s)}\le \sum_{k=1}^N
||f_k||_{p,\infty}
$$
which implies that
\begin{equation}\label{limiting5}
\varlimsup_{s\to\infty}||f||_{(p,s)}\le ||f||_{(p,\infty)}.
\end{equation}

To prove the reverse inequality, take an arbitrary
$\varepsilon>0.$ For a fixed $s>s_0,$  find a decomposition
(\ref{limiting3}) such that
$$
||f||_{(p,s)}> \sum_{k=1}^N ||f_k||_{p,s}-\varepsilon.
$$
Applying inequality  (\ref{lorentz}), we obtain
$$
\begin{aligned}
||f||_{(p,s)}&> \left(\frac{s}{p}\right)^{1/s} \sum_{k=1}^N
||f_k||_{p,\infty}-\varepsilon\\
&>\sum_{k=1}^N ||f_k||_{p,\infty}-\varepsilon\ge
||f||_{(p,\infty)}-\varepsilon.
\end{aligned}
$$
Thus, $ ||f||_{(p,s)}>||f||_{(p,\infty)}-\varepsilon, $ for any
$s>\s_0$ and any $\varepsilon>0.$ It follows that
$$
\varliminf_{s\to\infty}||f||_{(p,s)}\ge ||f||_{(p,\infty)},
$$
which, together with (\ref{limiting5}), proves (\ref{limiting2}).
\end{proof}\medskip

\begin{lem}\label{char} Let $h(x)=\chi_{[0,1]}(x)$ and  $1<p<\infty$, $1\le s\le \infty.$ Then
\begin{equation}\label{char1}
||h||_{p,s}=\left(\frac{p}{s}\right)^{1/s}.
\end{equation}
If $p<s,$ then
\begin{equation}\label{char3}
||h||'_{p,s}=||h||_{(p,s)}=\left(\frac{s'}{p'}\right)^{1/s'}.
\end{equation}
\end{lem}

\begin{proof}
The equality (\ref{char1}) is immediate. We shall prove
(\ref{char3}). Denote $\a=1-{s'}/{p'}$ and set
\begin{equation}\label{char4}
\f(t)=(1-\a)t^{-\a}, t\in (0,1].
\end{equation}
We have
\begin{equation}\label{char5}
||\varphi\chi_{[0,1]}||_{p,s}=\left(\frac{s'}{p'}\right)^{1/s'}.
\end{equation}
To evaluate the dual norm of $h$, we assume that $g\in
L^{p',s'}(\R_+),~~g\ge 0$ and $\|g\|_{p',s'}=1$. Applying
(\ref{cheb}), H\"older's inequality (\ref{holder}), and
(\ref{char5}), we obtain
$$
\begin{aligned}
\int_{\R_+}h(x)g(x)dx&\le \int_0^1 g^{*}(x)dx\\
&\le(1-\a)\int_0^1
g^{*}(x)x^{-\a}dx\le\|g\|_{p',s'}\|\f\|_{p,s}=\left(\frac{s'}{p'}\right)^{1/s'}.
\end{aligned}
$$
On the other hand, if,
$$
g(x)=\left(\frac{s'}{p'}\right)^{1/s'}\chi_{[0,1]}(x),
$$
then $\|g\|_{p',s'}=1$ and
$$
\int_{\R_+}h(x)g(x)dx=\left(\frac{s'}{p'}\right)^{1/s'}.
$$
Thus,
\begin{equation}\label{char51}
\|h\|_{p,s}^{'}=\left(\frac{s'}{p'}\right)^{1/s'}.
\end{equation}

We prove now the second  equality in
(\ref{char3}) (in Section 4 we shall
prove that the dual and the decomposition norms always agree, but  the proof of this fact for a characteristic function is
much simpler). Let $1<p<s<\infty.$ Assume that the function $\f$
in (\ref{char4}) is extended to the whole line $\R$ periodically
with   period $1$. Set
$$
g_N(x)=N\int_x^{x+1/N}\f(t)dt.
$$
Then $g_N(x)\rightarrow\f(x)$ as $N\rightarrow\infty$ for all
$x\in(0,1).$ Moreover,
$$
\left(g_N\chi_{[0,1]}\right)^{*}(t)\le
\left(g_N\chi_{[0,1]}\right)^{**}(t)\le \f^{**}(t)= t^{-\a},\quad
t\in(0,1].
$$
Applying Lebesgue's dominated convergence theorem and
(\ref{char5}), we obtain
$$
\|g_N\chi_{[0,1]}\|_{p,s}\rightarrow
\|\f\chi_{[0,1]}\|_{p,s}=\left(\frac{s'}{p'}\right)^{1/s'}.
$$
Let $\varepsilon >0$. Fix a number $N$ such that
\begin{equation}\label{char6}
\|g_N\chi_{[0,1]}\|_{p,s}<\left(\frac{s'}{p'}\right)^{1/s'}+\varepsilon
\end{equation}
Set
$$
f_k(x)=\int_{(k-1)/N}^{k/N}\f(x+t)dt=\frac{1}{N}g_N\left(x+\frac{k-1}{N}\right),\quad
k=1,\ldots,N.
$$
Then \begin{eqnarray}\label{char7}
 \sum_{k=1}^N f_k(x) &=& \int_0^1\f(x+t)dt=1,
\end{eqnarray}
for all $x$. Since $f_k$ are 1-periodic and
$f_k(x)=f_1\left(x+(k-1)/N)\right)$, the restrictions of $f_k$ to
$[0,1]$ are pairwise equimeasurable. Set now
$$
h_k(x)=f_k(x)\chi_{[0,1]}(x),\quad k=1,\ldots, N.
$$
Then, by (\ref{char7}), $
h=\sum_{k=1}^{N}h_k
$
and by (\ref{char6})
$$
\sum_{k=1}^{N}\|h_k\|_{p,s}<\left(\frac{s'}{p'}\right)^{1/s'}+\varepsilon.
$$
This implies that
\begin{equation}\label{char10}
\|h\|_{(p,s)}\le \left(\frac{s'}{p'}\right)^{1/s'},
\quad\text{for}\quad p<s<\infty.
\end{equation}
By Lemma \ref{limiting},   (\ref{char10}) holds for all $p<s\le
\infty.$
 The opposite inequality follows from
(\ref{char51}) and Lemma \ref{Dual-dec}.
\end{proof}\medskip

We shall  use the following Hardy's lemma \cite[p.\ 56]{BS}.
\begin{lem}\label{Hardy1}
Let $f_1$ and $f_2$ be nonnegative measurable functions on $\R_+$,
such that
$$
\int_0^tf_1(u)\,du\le \int_0^tf_2(u)\,du,
$$
for all $t>0.$ Then, for every  nonnegative and non-increasing function  $g$  on
$\R_+$, we have that
$$
\int_0^\infty f_1(u)g(u)\,du\le \int_0^\infty f_2(u)g(u)\,du.
$$
\end{lem}

\vskip 6pt

Finally, we recall the definition of the Hardy-Littlewood-P\'olya
relation. Let $(R,\mu)$ be a  measure space and let $f$ and $g$ be
$\mu-$measurable and $\mu-$a.e. finite functions on $R.$ We write
$f\prec g$ if
$$
\int_0^tf^*(u)\,du\le \int_0^tg^*(u)\,du,
$$
for all $t>0$ (see \cite[p. 55]{BS}).

\section{The level function}

\noindent
The notion of a level function was first introduced by Halperin
\cite{Hal}. We shall use the extension of this notion given by
Lorentz \cite{Lor} and based on the following theorem.

\begin{teo}\label{Halperin1} Let $\f$ be a positive measurable
 function on $\R_+$ such that
$$
\Phi(t)=\int_0^t\f(u)\,du<\infty,
$$
for all $t>0.$ Assume that $f$ is a nonnegative measurable
function on $\R_+$ and that
$$
\int_0^tf(u)du=o(\Phi(t)),\quad\text{as}\quad t\to \infty.
$$
Then, there exists a nonnegative  function $f^{\circ}$ on $\R_+$
satisfying the following conditions:
\begin{itemize}
\item[{(a)}] the function  $f^{\circ}(t)/\f(t)$ decreases on $\R_+$;

\item[{(b)}]  $f\prec f^\circ $;

\item[{(c)}]  up to a set of   measure zero, the set $\{t\in \R_+:
f(t)\not=f^\circ(t)\}$ is the union of bounded disjoint intervals
$I_k$ such that
$$
\int_{I_k}f(u)du=\int_{I_k}f^\circ(u)du,
$$
and $f^\circ(t)/\f(t)$ is constant on $I_k.$
\end{itemize}
\end{teo}
This theorem is a slight modification of the  results in \cite{Hal}
and \cite[\S 3.6]{Lor}; the proof is similar to the one given in
\cite[\S 3.6]{Lor} for functions defined on $[0,1]$. It is easy to
show that the function $f^\circ$ is uniquely determined (see
\cite[Theorem 3.7]{Hal}). It is called
 {\it {the level
function of $f$ with respect to}} $\f.$

\begin{teo}\label{D1}
Let $1<p<\infty$ and $p<s\le \infty.$ Suppose that $f\in
L^{p,s}(\R_+)$ is a nonnegative and non-increasing function on
$\R_+.$ Let $f^\circ$ be the level function of $f$ with respect to
the function $\f_0(t)=t^{-\a}, ~~\a=1-s'/p'$. Then
\begin{equation}\label{d7}
||f^\circ||_{p,s}\le ||f||_{p,s}\le c_{p,s}||f^\circ||_{p,s},
\end{equation}
where
\begin{equation}\label{d8}
c_{p,s}=\left(\frac{p}{s}\right)^{1/s}\left(\frac{p'}{s'}\right)^{1/s'}.
\end{equation}
The constants in the inequalities (\ref{d7}) are optimal.
\end{teo}
\begin{proof} First we assume that $s<\infty.$
We consider the left hand side inequality in (\ref{d7}). Applying
 Theorem \ref{Halperin1}(c), we have
$f^\circ(t)=\l_k t^{-\a}$ for all $t\in I_k$, where
$$
\l_k=\left(\int_{I_k}t^{-\a}\,dt\right)^{-1}\int_{I_k}f(t)\,dt.
$$
Since  $\a=(s/p-1)/(s-1),$ and $f^\circ(t)^{s-1}
t^{s/p-1}=\l_k^{s-1}$  then, applying H\"older's inequality, we
obtain
\begin{align}
\int_{I_k}f^\circ(t)^s
t^{s/p-1}\,dt&=\l_k^{s-1}\int_{I_k}f^\circ(t)\,dt=\left(\int_{I_k}t^{-\a}\,dt\right)^{1-s}\left(\int_{I_k}f(t)\,dt\right)^s\nonumber
\\
\label{inequality}
&\le
\int_{I_k}f(t)^s t^{s/p-1}\,dt.
\end{align}
This estimate and   property (c) yield the first inequality in
(\ref{d7}).

Now, denote
\begin{equation}\label{psi}
\psi(t)=f(t)^{s-1} t^{s/p-1}.
\end{equation}
Let $\widetilde\psi(t)$ be the level function of $\psi$ with
respect to $\f(t)=1.$ Applying Theorem \ref{Halperin1}, Lemma
\ref{Hardy1}, and the inequality (\ref{holder}), we obtain
$$
\begin{aligned}
||f||_{p,s}^s&=\int_0^\infty f(t)\psi(t)\,dt\le \int_0^\infty
f(t)\widetilde\psi(t)\,dt
\\
&\le \int_0^\infty f^\circ(t)\widetilde\psi(t)\,dt\le
||f^\circ||_{p,s}||\widetilde\psi||_{p',s'}.
\end{aligned}
$$
To obtain the second inequality in (\ref{d7}), it   suffices to
prove that
\begin{equation}\label{d9}
||\widetilde\psi||_{p',s'}\le c_{p,s}||f||_{p,s}^{s-1},
\end{equation}
where the constant $c_{p,s}$ is defined by (\ref{d8}).

Let $E=\{t\in \R_+: \widetilde\psi(t)=\psi(t)\}.$ Then, up to a set
of measure zero,
$$
\R_+\setminus E=\bigcup_k (a_k,b_k),
$$
where $(a_k,b_k)$ are bounded disjoint intervals  such that
\begin{equation}\label{star}
\widetilde\psi(t)=\frac{1}{b_k-a_k}\int_{a_k}^{b_k}\psi(u)du,
\quad\text{for all} \quad t\in (a_k,b_k).
\end{equation}
By   H\"older's inequality
$$
\begin{aligned}
\int_{a_k}^{b_k}\psi(u)du&\le
\left(\int_{a_k}^{b_k}u^{s/p-1}\,du\right)^{1/s}\left(\int_{a_k}^{b_k}f(u)^s
u^{s/p-1}\,du\right)^{1/s'}
\\
&=\left(\frac{p}{s}\right)^{1/s}(b_k^{s/p}-a_k^{s/p})^{1/s}\left(\int_{a_k}^{b_k}f(u)^s
u^{s/p-1}\,du\right)^{1/s'}.
\end{aligned}
$$
Using (\ref{star}) and applying Lemma \ref{Ineq}, we obtain that
$$
\widetilde\psi(t)\le
\left(\frac{p}{s}\right)^{1/s}(b_k^{s'/p'}-a_k^{s'/p'})^{-1/s'}\left(\int_{a_k}^{b_k}f(u)^s
u^{s/p-1}\,du\right)^{1/s'},
$$
for all $t\in (a_k,b_k).$ Thus,
$$
\begin{aligned}
\int_{a_k}^{b_k}\widetilde\psi(t)^{s'}t^{s'/p'-1}\,dt
&\le
\left(\frac{p}{s}\right)^{s'/s}(b_k^{s'/p'}-a_k^{s'/p'})^{-1}\\
&\qquad\times \int_{a_k}^{b_k}f(t)^s
t^{s/p-1}\,dt\int_{a_k}^{b_k}t^{s'/p'-1}\,dt
\\
&=\left(\frac{p}{s}\right)^{s'/s}\frac{p'}{s'}\int_{a_k}^{b_k}f(t)^s
t^{s/p-1}\,dt.
\end{aligned}
$$
We also  have that
$$
\begin{aligned}
\int_E\widetilde\psi(t)^{s'}t^{s'/p'-1}\,dt&=\int_E\psi(t)^{s'}t^{s'/p'-1}\,dt
\\
&= \int_Ef(t)^s t^{s/p-1}\,dt.
\end{aligned}
$$
Since
$$
c_{p,s}=\left(\frac{p}{s}\right)^{1/s}\left(\frac{p'}{s'}\right)^{1/s'}>1,
$$
we obtain (\ref{d9}).
Thus, the inequalities in (\ref{d7}) are proved for $s<\infty.$

Let
now $s=\infty$ and hence  $\a=1/p.$ For any $k,$
$$
\begin{aligned}
p'(b_k^{1/p'}-a_k^{1/p'})\l_k&=\int_{a_k}^{b_k}f^\circ(t)dt=\int_{a_k}^{b_k}f(t)dt
\\
&\le
||f||_{p,\infty}\int_{a_k}^{b_k}t^{-1/p}dt=p'(b_k^{1/p'}-a_k^{1/p'})||f||_{p,\infty}.
\end{aligned}
$$
Thus, $\l_k\le ||f||_{p,\infty}$, which implies that
$||f^\circ||_{p,\infty}\le ||f||_{p,\infty}.$ On the other hand,
for any $t\in (a_k,b_k)$ we have (see Theorem \ref{Halperin1} (b))
$$
t^{1/p}f(t)\le t^{1/p-1}\int_0^t f(u)du\le t^{1/p-1}\int_0^t
f^\circ(u)du\le p'||f^\circ||_{p,\infty}.
$$
This implies the second inequality in (\ref{d7}) for $s=\infty.$

The left hand side inequality in (\ref{d7}) becomes equality for
$f(t)=t^{-\a}\chi_{[0,1]}(t).$ Further, let $f=\chi_{[0,1]}.$ Then
$$
||f||_{p,s}=\left(\frac{p}{s}\right)^{1/s}.
$$
Next, $f^\circ(t)=(1-\a)t^{-\a}\chi_{[0,1]}(t),$
$$
||f^\circ||_{p,s}=\left(\frac{s'}{p'}\right)^{1/s'},
$$
and we have equality $||f||_{p,s}=c_{p,s}||f^\circ||_{p,s}.$ Thus,
the constants in (\ref{d7}) are optimal.
\end{proof}\medskip

\begin{rem}\label{level} Let $1<p<s\le \infty.$ Let
$f\in L^{p,s}(\R_+)$ be a nonnegative and non-increasing function
on $\R_+$ and let $f^\circ$ be the level function of $f$ with
respect to the function $\f_\a(t)=t^{-\a}~~(\a=1-s'/p')$. Then,
the equality
\begin{equation}\label{equality}
||f^\circ||_{p,s}=||f||_{p,s}
\end{equation}
 holds if and only if $f^\circ(t)=f(t)$, except for a countable set of points $t.$ Indeed, the last inequality in
(\ref{inequality}) becomes equality if and only if $f(t)t^\a$ is
constant on $I_k.$

In other words, (\ref{equality}) holds if and only if $f(t)t^\a$
decreases on $\R_+.$
\end{rem}

\section{The dual norm}

\noindent
Recall that for a function $f\in L^{p,s}(R,\mu)~~~(1< p <\infty,
1\le s\le \infty)$ its dual norm is defined by
\begin{equation}\label{DUAL}
||f||_{p,s}'=\sup\left\{\int_R fg\,d\mu:\quad
||g||_{p',s'}=1\right\},
\end{equation}
where the supremum is taken over all functions $g\in
L^{p',s'}(R,\mu)$ with $||g||_{p',s'}=1.$

By Lemma \ref{dual} and the Hardy-Littlewood inequality \cite[p.\
44]{BS}, for any function $f\in L^{p,s}(R,\mu)~~~(1< p <\infty,
1\le s\le \infty)$
\begin{equation}\label{DUAL1}
||f||_{p,s}'=\sup\left\{\int_0^\infty f^*(t)g(t)\,dt:
||g||_{p',s'}=1\right\},
\end{equation}
where the supremum is taken over all nonnegative and nonincreasing
functions $g\in L^{p',s'}(\mathbb R_+)$ with $||g||_{p',s'}=1.$

Suppose that  $1<p<\infty$ and $1\le s\le \infty.$ Let $f\in
L^{p,s}(\R_+)$ and let $g\in L^{p',s'}(\R_+).$ By H\"older's
inequality (\ref{holder})
\begin{equation}\label{d1}
\int_0^\infty |f(t)g(t)|\,dt\le ||f||_{p,s}||g||_{p',s'}.
\end{equation}
 It follows  that
\begin{equation}\label{d3}
||f||'_{p,s}\le ||f||_{p,s}.
\end{equation}
If $s\le p,$ then we have the equality of norms
\begin{equation}\label{d4}
||f||'_{p,s}= ||f||_{p,s}.
\end{equation}
Indeed,
$$
||f||_{p,s}^s=\int_0^\infty f^*(t)\psi(t)\,dt, \quad
\psi(t)=f^*(t)^{s-1}t^{s/p-1}.
$$
If $s\le p,$ then the function $\psi$ is non-increasing and we
have
$$
||\psi||^{s'}_{p',s'}=\int_0^\infty
\psi(t)^{s'}t^{s'/p'-1}\,dt=||f||_{p,s}^s.
$$
The latter two equalities imply that $||f||'_{p,s}\ge
||f||_{p,s}.$ Together with (\ref{d3}) this yields (\ref{d4}).
Observe also that the supremum in (\ref{DUAL1}) is attained on the
function $g(t)=\psi(t)/||\psi||_{p',s'}.$

Now we assume that $p<s\le\infty.$ Let $f\in L^{p,s}(\R_+).$ If
the function $f^*(t)t^{1-s'/p'}$ is non-increasing, then as above
we have the equality  (\ref{d4}). Let $f$ be an arbitrary
nonnegative function in $L^{p,s}(\R_+)$ and  let $g\in
L^{p',s'}(\R_+), ~~g\ge 0,$ be a nonincreasing function. By Lemma
\ref{Hardy1}, we have that
\begin{equation}\label{d5}
\int_0^\infty f(t)g(t)\,dt\le \inf_{f\prec h}||h||_{p,s}||g||_{p',s'}.
\end{equation}
 This implies that
\begin{equation}\label{d50}
||f||'_{p,s}\le \inf_{f\prec h}||h||_{p,s}.
\end{equation}
Note that in the case $s\le p$ the  infimum in (\ref{d50}) is
equal to $||f||_{p,s}.$ However, for $s>p$ the infimum may be
smaller than $||f||_{p,s}$ and (\ref{d5}) may give  a refinement
of the inequality (\ref{d1}). It was proved by Halperin
\cite[Theorem 4.2]{Hal} (see also \cite[Theorem 3.6.5]{Lor}) that
  equality in (\ref{d50}) holds  and the infimum is
attained for some $h\in L^{p,s}(\R_+).$ Since the proofs given in \cite{Hal} and \cite{Lor} do not cover
explicitly the case $s=\infty$, and for the sake of completeness, we show the result for all $p<s\le \infty.$

\begin{teo}\label{Halperin2}
Let $1<p<s\le \infty.$ Assume that $f\in L^{p,s}(\R_+)$ is a
nonnegative and non-increasing function on $\R_+.$ Set
$\a=1-s'/p'$ and $\f_\a(t)=t^{-\a}.$ Then
\begin{equation}\label{d6}
||f||'_{p,s}= \inf_{f\prec h}||h||_{p,s}=||f^\circ||_{p,s},
\end{equation}
where $f^\circ$ is the level function of $f$ with respect to the
function $\f_\a.$
\end{teo}
\begin{proof}  In view of (\ref{d50}) and Theorem \ref{Halperin1} (b), it
suffices to prove that
\begin{equation}\label{d100}
||f||'_{p,s}\ge ||f^\circ||_{p,s}.
\end{equation}

Set
$$
 E=\{x\in \R_+: f(x)=f^\circ(x)\}.
$$
By Theorem \ref{Halperin1}, up to a set of measure zero,
$$
\mathbb R^+\setminus E=\bigcup_k(a_k,b_k),
$$
where $(a_k,b_k)$ are disjoint bounded intervals such that
\begin{equation}\label{101}
\int_{a_k}^{b_k}f(t)\,dt=\int_{a_k}^{b_k}f^\circ(t)\,dt.
\end{equation}

We first assume that $s<\infty.$ Denote
$\psi(t)=f^\circ(t)^{s-1}t^{s/p-1}.$ As above, we have
$$
||\psi||^{s'}_{p',s'}=\int_0^\infty
\psi(t)^{s'}t^{s'/p'-1}\,dt=||f^\circ||_{p,s}^s.
$$
Set $g(t)=\psi(t)/||f^\circ||_{p,s}^{s-1}.$ Then
$||g||_{p',s'}=1.$ For each $k,$ we have $ f^\circ(t)=\lambda_k
t^{-\a}$ {and} $\psi(t)=\lambda_k^{s-1}$, {for} $ t\in (a_k,b_k) $
(where $\lambda_k$ is a constant). Thus,
$$
\begin{aligned}
||f^\circ||_{p,s}^{s-1}\int_{a_k}^{b_k}f(t)g(t)dt&=\lambda_k^{s-1}\int_{a_k}^{b_k}f(t)dt
\\
&=\lambda_k^{s-1}\int_{a_k}^{b_k}f^\circ(t)dt=
\int_{a_k}^{b_k}[t^{1/p}f^\circ(t)]^s\frac{dt}{t}.
\end{aligned}
$$
Besides, we have
$$
||f^\circ||_{p,s}^{s-1}\int_E f(t)g(t)dt=\int_E
[t^{1/p}f^\circ(t)]^s\frac{dt}{t},
$$
and thus,
$$
\int_0^\infty f(t)g(t)dt=||f^\circ||_{p,s},
$$
from which we obtain (\ref{d100}).

Let now $s=\infty.$ In this case we have
\begin{equation}\label{d102}
||f^\circ||_{p,\infty}=\lim_{t\to 0+} f^\circ(t)t^{1/p}.
\end{equation}

We assume first that for some $k$ we have $a_k=0.$ Set
$$g(t)=\chi_{(0,b_k)}/(p'b_k^{1/p'}).$$
 Then $||g||_{p',1}=1.$ We
have
$$
f^\circ(t)=\lambda_k t^{-1/p}\quad\text{for}\quad t\in
(0,b_k)\quad\text{and}\quad ||f^\circ||_{p,\infty}=\lambda_k.
$$
Thus,
$$
\begin{aligned}
\int_0^\infty f(t)g(t)dt&=(p'b_k^{1/p'})^{-1}\int_0^{b_k}f(t)dt
\\
&=(p'b_k^{1/p'})^{-1}\int_0^{b_k}f^\circ(t)dt=\lambda_k=||f^\circ||_{p,\infty}.
\end{aligned}
$$
This implies (\ref{d100}).

Now we assume that $a_k\not=0$ for each $k.$ Then, for any $\d>0$
we have
\begin{equation}\label{new1}
(0,\d)\cap A\not=\emptyset, \quad\text{where}\quad A=\R_+\setminus
\cup_j(a_j,b_j).
\end{equation}
On the other hand, by Theorem \ref{Halperin1} (c), for any $t\in
A$
\begin{equation}\label{new}
\int_0^t f(u)du=\int_0^t f^\circ(u)du.
\end{equation}
 Let $\varepsilon>0.$ By (\ref{d102}),
there exists  $\d>0$ such that
$$
f^\circ(t)t^{1/p}> ||f^\circ||_{p,\infty}-\varepsilon
\quad\text{for any}\quad t\in (0,\d).
$$
Let $\xi\in (0,\d)\cap A.$ Set
$g(t)=\chi_{(0,\xi)}/(p'\xi^{1/p'}).$ Then $||g||_{p',1}=1.$
Applying (\ref{new}) and (\ref{new1}), we get
$$
\int_0^\infty f(t)g(t)dt=(p'\xi^{1/p'})^{-1}\int_0^\xi
f^\circ(t)dt>||f^\circ||_{p,\infty}-\varepsilon,
$$
which again implies (\ref{d100}).
\end{proof}\medskip

\begin{rem}
Note that for $1<p<s< \infty$ the supremum in (\ref{DUAL1}) is
attained on the function $g(t)=\psi(t)/||\psi||_{p',s'},$ where
$\psi(t)=f^\circ(t)^{s-1}t^{s/p-1}.$ If $s=\infty,$ then the
supremum in (\ref{DUAL1}) may not  be attained.
\end{rem}

\begin{rem}\label{equds} Let $1<p<s\le \infty,$ and let  $f\in L^{p,s}(\R_+)$
be a nonnegative and non-increasing function on $\R_+.$ Then, by
Remark \ref{level}, the equality
$$
||f||'_{p,s}=||f||_{p,s}
$$
holds if and only if $f(t)t^\a$ decreases on $\R_+.$
\end{rem}

The following theorem gives the sharp estimate of the standard
norm via the dual norm.
\begin{teo}\label{D2}
Let $1<p<\infty$ and $p<s\le \infty.$ Then, for any function $f\in
L^{p,s}(R,\mu)$
\begin{equation}\label{d10}
||f||_{p,s}\le c_{p,s}||f||'_{p,s},
\end{equation}
where
 $$
c_{p,s}=\left(\frac{p}{s}\right)^{1/s}\left(\frac{p'}{s'}\right)^{1/s'}.
$$
The constant $c_{p,s}$ is optimal.
\end{teo}
This theorem follows immediately from Theorems \ref{Halperin2} and
\ref{D1}. However,  a direct proof can be given exactly as in
Theorem \ref{D1}. Indeed,  assume that $f$ is nonnegative and
non-increasing on $\R_+.$ As in the proof of Theorem \ref{D1}, we
have
$$
\begin{aligned}
||f||_{p,s}^s&=\int_0^\infty f(t)\psi(t)\,dt\le \int_0^\infty
f(t)\widetilde\psi(t)\,dt
\\
&\le ||f||'_{p,s}||\widetilde\psi||_{p',s'}.
\end{aligned}
$$
Applying the inequality (\ref{d9}), we obtain (\ref{d10}). Let now
 $f=\chi_{[0,1]}.$ Then, by Lemma \ref{char}
$$
||f||'_{p,s}=\left(\frac{s'}{p'}\right)^{1/s'}
\quad\text{and}\quad||f||_{p,s}=\left(\frac{p}{s}\right)^{1/s}
=c_{p,s}||f||'_{p,s},
$$
which shows that the constant in (\ref{d10}) is optimal.

\section{The decomposition norm }

\noindent
 In this section we prove one of the main results of this paper --the coincidence of the dual   and the decomposition norms. The following lemma
plays an important role in the proof of the equality of these two norms.

\begin{lem}\label{matrix} Let $\alpha_1\ge\alpha_2\ge\cdots\ge\alpha_{\nu}$ be positive numbers and let $\{\eta_{jk}\}$ be a
$(N\times\nu)-$matrix of  positive numbers $(1\le j\le N,\, 1\le
k\le \nu).$ Set
$$
\beta_k=\sum_{j=1}^N\eta_{jk},\qquad k=1,\cdots,\nu.
$$
Assume that
\begin{equation}\label{4.1}
\beta_1+\cdots+\beta_k\ge\alpha_1+\cdots+\alpha_k,
\end{equation}
for any $k=1,\dots,\nu.$ Let $\eta=\max\eta_{jk}.$ Then, for any
$j=1,\dots,N$ there exists a  permutation
$\{\tilde\eta_{jk}\}_{k=1}^{\nu}$ of the $\nu-$tuple
$\{\eta_{jk}\}_{k=1}^{\nu}$ such that
\begin{equation}\label{4.2}
\alpha_k\le\tilde\beta_k+\eta,\qquad
\tilde\beta_k=\sum_{j=1}^N\tilde\eta_{jk},
\end{equation}
for any $k=1,\dots,\nu.$
\end{lem}
\begin{proof}
For $\nu=1$ the lemma is obvious. Assume that it is true for
$\nu-1~~(\nu\ge 2).$ We have $\b_1\ge\a_1.$ If
 $\beta_k\ge \alpha_1$ for all   $k=1,\dots,\nu$, there is nothing to prove. Otherwise,
 denote by $s$ the least natural
  $k$ for which
$\beta_k<\alpha_1$. Then, $s\ge 2.$ Set $\gamma_0=\beta_1$,
$\gamma_N=\beta_s,$ and
$$
\gamma_m=\sum_{j=1}^m\eta_{js}+\sum_{j=m+1}^N\eta_{j1},
\quad\mbox{for}\quad1\le m<N.
$$
We have $\gamma_0\ge\alpha_1$ and $\gamma_N<\alpha_1.$ Let $m_0$
be the least $m$ for which $\gamma_m<\alpha_1$. Since
$|\gamma_m-\gamma_{m-1}|\le\eta$ for any $m=1,\cdots,N$, we have
that
\begin{equation}\label{4.3}
\gamma_{m_0}<\alpha_1\le\gamma_{m_0-1}\le\gamma_{m_0}+\eta.
\end{equation}
Set
$$
\begin{aligned}
\tilde\eta_{j1}&=\begin{cases}
     \eta_{js} & \text{if } 1\le j\le m_0\\
         \eta_{j1} & \text{ if } m_0< j\le N,
\end{cases}
\\
\eta_{js}'&=\begin{cases}
     \eta_{j1} & \text{ if} ~~1\le j\le m_0\\
         \eta_{js} & \text{ if } m_0+1< j\le N,
\end{cases}
\end{aligned}
$$
and $\eta_{jk}'=\eta_{jk}~~~(j=1,...,N)$, if $k\not= 1,s.$ Using  (\ref{4.3}) we have
\begin{equation}\label{4.4}
\tilde\beta_1<\alpha_1\le\tilde\beta_1+\eta,\quad\mbox{where}\quad
\tilde\beta_1=\gamma_{m_0}= \sum_{j=1}^N\eta_{js}'.
\end{equation}
Denote also $ \beta_k'=\beta_k, \quad k=2,..., N. $ We first
assume that $s=2$. We have
$$
\beta_1+\beta_2+\cdots+\beta_k\ge\alpha_1+\alpha_2+\cdots+\alpha_k,
$$
for each $k\ge 2.$ But  $\beta_1+\beta_2=\tilde\beta_1+\beta_2'$
and $\tilde\beta_1<\alpha_1$, by (\ref{4.4}). Thus,
$$
\beta_2'+ \cdots+\beta'_k\ge\alpha_2+\cdots+\alpha_k,\qquad
k=2,\dots,\nu.
$$
Now we assume that $s>2$. Then we have for every $2\le l<s$
$$
\beta_l'=\beta_l\ge\alpha_1\ge\alpha_l,
$$
and therefore
$$
\beta_2'+ \cdots+\beta_k'\ge\alpha_2+\cdots+\alpha_k,\qquad 2\le
k<s.
$$
Let $k\ge s.$ Since
$$
\tilde\beta_1+\beta_2'+ \cdots+\beta_k'=\beta_1+\beta_2+ \cdots+
\beta_k,
$$
and $\tilde\beta_1<\alpha_1$ (see (\ref{4.4})), it follows from
(\ref{4.1}) that
$$
\beta_2'+ \cdots+\beta_k'\ge \alpha_2+\cdots+\alpha_k.
$$
Thus, we can apply our inductive assumption to the
$(N\times(\nu-1))$-matrix
$$
\{\eta_{jk}'\}, \quad 1\le j\le N,~~2\le k\le \nu.
$$
Together with (\ref{4.4}), this proves the lemma.
\end{proof}\medskip

\begin{teo}\label{Main} Let $1< p<\infty$ and $1\le s\le
\infty.$ Then, for any function $f\in L^{p,s}(R,\mu)$
\begin{equation}\label{main}
||f||'_{p,s}=||f||_{(p,s)}.
\end{equation}
\end{teo}
\begin{proof} If $s\le p,$ then
$$
||f||'_{p,s}=||f||_{p,s}=||f||_{(p,s)}.
$$
We assume that $1<p<s\le\infty.$ By Lemma \ref{Dual-dec},
$$
||f||'_{p,s}\le ||f||_{(p,s)}.
$$
We shall prove that
\begin{equation}\label{main1}
||f||_{(p,s)}\le ||f||_{p,s}'.
\end{equation}

By virtue of (\ref{2.1}) and (\ref{decomp2}), it suffices to prove
(\ref{main1}) in the case when $(R,\mu)$ is $\R_+$, with Lebesgue's
measure, and $f$ is a nonnegative and non-increasing function on
$\R_+$. Applying      Lemma \ref{Decomp}(3), we can also
assume that there exist $0<x_0<x_1<\infty$ such that $f(x)=c_0>0$
on $(0,x_0)$ and $f(x)= 0$ for all $x>x_1.$

By Theorem \ref{Halperin2},
$$
||f||'_{p,s}= ||f^\circ||_{p,s},
$$
where $f^\circ$ is the level function of $f$ with respect to the
function $\f_\a(t)=t^{-\a}, ~~\alpha=1-s'/p'.$ Set
$$
 E=\{x\in \R_+:
f(x)=f^\circ(x)\}.
$$
By Theorem \ref{Halperin1}, up to a set of measure zero,
$$
\mathbb R^+\setminus E= \bigcup_i(a_i,b_i),
$$
where $(a_i,b_i)$ are disjoint bounded intervals such that
\begin{equation}\label{main2}
\int_{a_i}^xf(t)\,dt\le\int_{a_i}^xf^\circ(t)\,dt,\qquad
x\in(a_i,b_i)
\end{equation}
and
\begin{equation}\label{main3}
\int_{a_i}^{b_i}f(t)\,dt=\int_{a_i}^{b_i}f^\circ(t)\,dt.
\end{equation}

By our assumption, $f(x)=c_0$ on $(0,x_0).$ At the same time,
$f^\circ$ is strictly decreasing on $(0,x_0).$ This implies that,
for some $i$ we have $a_i=0.$ Indeed, assume the contrary. Then,
as   is easily seen,
 there exists $(a_k,b_k)$ such that $0<a_k<b_k< x_0.$ We
have $f^\circ(x)=\l_kx^{-\a}$ on $(a_k,b_k)$.
 Further,
$$
c_0a_k=\int_0^{a_k}f^\circ(x)dx> \frac{\l_k a_k^{1-\a}}{1-\a},
$$
and therefore
$$
\l_k< c_0(1-\a)a_k^\a.
$$
From here,
$$
\begin{aligned}
c_0b_k&=\int_0^{b_k}f^\circ(x)dx=
c_0a_k+\int_{a_k}^{b_k}f^\circ(x)dx
\\
&= c_0a_k+ \frac{\l_k }{1-\a}(b_k^{1-\a}-a_k^{1-\a})< c_0b_k.
\end{aligned}
$$

Thus, we can assume that $a_1=0.$ Let $b=\max(x_1, \sup_j b_j).$
Then $b<\infty$ and $f^\circ(x)=0$ for all $x>b.$

Let $\varepsilon>0.$ For any  $\nu\in\mathbb N$, define the
function $g_{\nu}$ in the following way.  First, set
$g_{\nu}(x)=f(x)$ for $x\in E;$ then $g_{\nu}(x)=0$ for all $x>b.$
Further, we subdivide each interval $(a_i,b_i)$ into $\nu$
subintervals $\Delta_k^i$, $k=1,\cdots,\nu$, of   length
$|\Delta_k^i|=(b_i-a_i)/\nu$, and set
$$
g_{\nu}(x)=
       |\Delta_k^i|^{-1}\int_{\Delta_k^i}f(t)\,dt \quad \text{for}\quad x\in \Delta_k^i,\quad k=1,...,\nu.
$$
It is easy to see that there exists $\nu_1$ such that
\begin{equation}\label{main4}
\Vert f-g_{\nu}\Vert_{p,s}<\varepsilon,
\end{equation}
for all $\nu\ge \nu_1.$ It follows that, for all $\nu\ge \nu_1$,
\begin{align}
\Vert f\Vert_{(p,s)}&\le\Vert g_{\nu}\Vert_{(p,s)}+\Vert f-
g_{\nu}\Vert_{p,s}\nonumber
\\ \label{main5}
&\le\Vert g_{\nu}\Vert_{(p,s)}+\varepsilon.
\end{align}

Similarly, for every $\nu\in\mathbb N$ we  define the function
$\psi_{\nu}$ approximating $f^\circ.$ Set
$\psi_{\nu}(x)=f^\circ(x)$ for $x\in E$ and
$$
\psi_{\nu}(x)=
       |\Delta_k^i|^{-1}\int_{\Delta_k^i}f^\circ(t)\,dt, \quad \text{for}\quad x\in \Delta_k^i,\quad k=1,...,\nu.
$$
There exists an integer $\nu_2\ge\nu_1$ such that
\begin{equation}\label{main6}
\Vert \psi_{\nu}\Vert_{p,s}\le\Vert
f^\circ\Vert_{p,s}+\varepsilon,
\end{equation}
for all $\nu\ge\nu_2$.
Fix $\nu\ge\nu_2.$ Next, choose a number $\delta>0$ such that
\begin{equation}\label{main7}
\delta<\varepsilon b^{-1/p}\left(\frac{s}{p}\right)^{1/s}.
\end{equation}
We shall prove that there exist a number $N\in\mathbb N$ and
functions $f_j\ge 0$, $j=1,\dots,N$, such that
\begin{equation}\label{main8}
g_{\nu}(x)\le\sum_{j=1}^Nf_j(x)+\delta,\qquad x>0,
\end{equation}
and
\begin{equation}\label{main9}
\Vert f_j\Vert_{p,s}=\Vert\psi_{\nu}\Vert_{p,s}/N,~~~~j=1,\dots,
N.
\end{equation}

For any $i,$ denote
$$
\beta^{(i)}_k= |\Delta_k^i|^{-1}\int_{\Delta_k^i}f^\circ(t)\,dt,
~~~k=1,...,\nu.
$$
There exists a number $N'\in \N$ such that
$$
\beta^{(1)}_k<N'\delta,\qquad k=1,\dots,\nu.
$$
On the other hand, since  $f^\circ$ is bounded on
$[b_{1},\infty)$, there exists $N''\in \N$ such that, for all
$i\ge 2$,
$$
\beta^{(i)}_k<N''\delta,\qquad k=1,\dots,\nu.
$$
Let $N=\max(N',N'').$  Then, for any $i$
$$
\beta^{(i)}_k<N\delta,\qquad k=1,\dots,\nu.
$$

Now we define the functions $f_j$, $j=1,\dots,N$. Set
$$
f_j(x)=\frac{1}{N}\psi_{\nu}(x),\quad \text{for}\quad x\in E.
$$
Further, consider an interval $(a_i,b_i).$ Set
$$
\eta_{jk}^{(i)}=\frac{\beta_k^{(i)}}{N}, \quad j=1,\dots,N,\quad
k=1,\ldots,\nu.
$$
Let
$$
\a_k^{(i)}=|\Delta_k^{i}|^{-1}\int_{\Delta_k^{i}}f(t)dt.
$$
Then, by (\ref{main2})
$$
\beta_1^{(i)}+\cdots+\beta_k^{(i)}\ge
\a_1^{(i)}+\cdots+\a_k^{(i)},\quad k=1,\ldots,\nu.
$$
Applying Lemma \ref{matrix}, we obtain that, for any fixed $i$ and
every $j=1,\ldots, N$, there exists a permutation
$\left\{\tilde\eta_{jk}^{(i)}\right\}_{k=1}^{\nu}$ of the
$\nu$-tuple $\left\{\eta_{jk}^{(i)}\right\}_{k=1}^{\nu}$ such that
\begin{equation}\label{main10}
\a_k^{(i)}\le \sum_{j=1}^{N}\tilde\eta_{jk}^{(i)}+\delta,
\end{equation}
for any $k=1,\ldots,\nu$. Set now
$$
f_j(x)=\tilde\eta_{jk}^{(i)},\quad\text{for}\quad x\in
\Delta_k^{i}.
$$
The functions $f_j ~~(j=1,\dots,N$) are defined on $\R_+$ and each
of them is equimeasurable with $\psi_{\nu}/N$. Thus, we have
(\ref{main9}). Moreover, (\ref{main10}) implies (\ref{main8}).
Applying (\ref{main7})--(\ref{main9}), and taking into account
that $g_{\nu}(x)=0$ for $x>b,$ we obtain
$$
\|g_{\nu}\|_{(p,s)}\le
\sum_{j=1}^N\|f_{j}\|_{p,s}+\delta||\chi_{[0,b]}||_{p,s} \le
\|\psi_{\nu}\|_{p,s}+2\e.
$$
Using (\ref{main5}) and (\ref{main6}), we obtain
$$
\|f\|_{(p,s)}\le \|f^{\circ}\|_{p,s}+4\e.
$$
This implies (\ref{main1}).
\end{proof}\medskip

\begin{cor} Let $f\in L^{p,s}(R,\mu)~~(1\le p<\infty,~~1\le s\le
\infty).$ Then
\begin{equation}\label{4.5}
||f||_{(p,s)}=||f^*||_{(p,s)}.
\end{equation}
\end{cor}
Indeed, (\ref{4.5}) follows immediately from (\ref{2.1}) and
(\ref{main}). Observe that (\ref{4.5}) does not follow directly
from the definition.

\section{The triangle inequality}

\noindent
Applying Theorem \ref{D2} and Theorem \ref{Main}, we immediately
obtain the following version of the \lq\lq triangle inequality."

\begin{teo}\label{Triangle} Let $1<p<s\le \infty.$ Assume that
$f_k\in L^{p,s}(R,\mu) ~~~~(k=1,...,N).$ Then
\begin{equation}\label{triangle}
\left\|\sum_{k=1}^N f_k \right\|_{p,s}\le
c_{p,s}\sum_{k=1}^N||f_k||_{p,s},
\end{equation}
where
$$
c_{p,s}=\bigg(\frac{p}{s}\bigg)^{1/s}\bigg(\frac{p'}{s'}\bigg)^{1/s'},
$$

and the constant is optimal.
\end{teo}

\begin{rem}

 It is clear that (\ref{triangle}) is
equivalent to the inequality
\begin{equation}\label{triangle1}
||f||_{p,s}\le c_{p,s}||f||_{(p,s)},
\end{equation}
where $f$ is any function in  $L^{p,s}(R,\mu).$ Inequality
 (\ref{triangle1}) follows directly from
(\ref{d10}) and Lemma \ref{Dual-dec}. By Lemma \ref{char},
(\ref{triangle1}) becomes equality for $f=\chi_{[0,1]}.$ Thus,
Theorem  \ref{Triangle} follows from Theorem \ref{D2} and Lemmas
\ref{Dual-dec} and \ref{char}.
\end{rem}

We also have the following continuous version of the  Minkowski type
inequality.

\begin{teo}\label{Minkowski} Suppose that $(R,\mu)$ is a $\s$-finite nonatomic measure space and $(Q,\nu)$ is a $\s$-finite
measure space. Let $f$ be a nonnegative measurable function on
$(R\times Q, \mu\times\nu).$  Assume that $1<p<s\le \infty$ and
that, for almost all $y\in Q$ the function
$$
f_y(x)=f(x,y),\quad x\in R,
$$
belongs to $L^{p,s}(R,\mu).$ Set $F(x)=\int_Q f(x,y)d\nu(y),
~~~x\in R.$ Then
\begin{equation}\label{minkowski}
||F||_{p,s}\le c_{p,s}\int_Q ||f_y||_{p,s}d\nu(y),
\end{equation}
where
\begin{equation}\label{6.0}
c_{p,s}=\bigg(\frac{p}{s}\bigg)^{1/s}\bigg(\frac{p'}{s'}\bigg)^{1/s'},
\end{equation}
and the constant is optimal.
\end{teo}
\begin{proof}
By Theorem \ref{D2},
\begin{equation}\label{6.1}
||f||_{p,s}\le
c_{p,s}||f||'_{p,s},
\end{equation}
where the constant $c_{p,s}$ is defined by (\ref{6.0}). Let $g\in
L^{p',s'}(R,\mu)$ and assume that $||g||_{p',s'}=1.$ Applying
Fubini's Theorem and H\"older's inequality, we obtain
$$
\begin{aligned}
\int_R F(x)g(x)d\mu(x)&=\int_R \left(\int_Q
f(x,y)d\nu(y)\right)g(x)d\mu(x)
\\
&=\int_Q \int_R f(x,y)g(x)d\mu(x)d\nu(y)\le \int_Q
||f_y||_{p,s}d\nu(y).
\end{aligned}
$$
Together with (\ref{6.1}), this implies (\ref{minkowski}).
Finally, it follows from Theorem~\ref{Triangle} that the constant
$c_{p,s}$ in (\ref{minkowski}) cannot be replaced by a smaller
one.
\end{proof}

\vskip 1cm

\noindent
Sorina Barza\\ Dept.\  of Mathematics\\ Karlstad University\\
SE-65188  Karlstad, Sweden \ \ \ {\sl E-mail:}
 {\tt sorina.barza@kau.se}

\medskip

\noindent
Viktor Kolyada\\ Dept.\  of Mathematics\\ Karlstad University\\
SE-65188 Karlstad, Sweden \ \ \ {\sl E-mail:}
 {\tt viktor.kolyada@kau.se}

\medskip

\noindent
Javier Soria\\ Dept.\  Appl.\  Math.\  and Analysis
\\ University of Barcelona\\ E-08007 Barcelona,
 Spain\ \ \ {\sl E-mail:}
 {\tt soria@ub.edu}


\medskip

\begin{thebibliography}{99}



\bibitem{BS} C. Bennett and R. Sharpley,  {\it Interpolation of
Operators,} Academic Press, Boston 1988.

\bibitem{CRS}  M.J. Carro, J.A. Raposo, and J. Soria, {\it
Recent Developments in the Theory of Lorentz Spaces and  Weighted
Inequalities}, Mem. Amer. Math. Soc. \textbf{187},  Providence, RI, 2007.

\bibitem{CS} M.J. Carro and J. Soria, {\it Weighted Lorentz
spaces and the Hardy operator}, J. Funct. Anal. \textbf{112}
(1993), 480--494.

\bibitem{DiB} E. DiBenedetto, {\it Real Analysis}, Birkh\"auser,
Boston 2002.

\bibitem{Hal} I. Halperin, {\it Function spaces,} Canad. J. Math.
{\bf 5} (1953), 273 -- 288.

\bibitem{HM} H.P. Heinig and L. Maligranda, {\it Chebyshev
inequality in function spaces,} Real Anal. Exchange, {\bf 17}
(1991-92), 211--247.

\bibitem{Hu} R.A.\ Hunt, {\it On $L(p,\,q)$,} Enseignement Math.\
{\bf 12} (1966), 249--276.

\bibitem{ko} V.I. Kolyada, {\it Rearrangement of functions and embedding of
anisotropic spaces of Sobolev type}, East J.\ Approx. {\bf 4}
(1998), no. 2, 111 -- 199.


\bibitem{K} V.I. Kolyada, {\it Inequalities of Gagliardo-Nirenberg
type and estimates for the moduli of continuity}, Russian Math.\
Surveys \textbf{60} (2005), 1147--1164.

\bibitem{loran} G.G. Lorentz, {\it Some new functional spaces,}  Ann.\ of Math.\  {\bf 51}
(1950), 37--55.

\bibitem{lorpjm} G.G. Lorentz, {\it On the theory of spaces $\Lambda$,}  Pacific J.\ Math.\   {\bf 1}
(1951), 411--429.


\bibitem{Lor} G.G. Lorentz, {\it Bernstein polynomials,} Univ of
Toronto Press, Toronto, 1953.

\bibitem{Sw} E. Sawyer,  {\it Boundedness of classical operators on classical
Lorentz spaces}, Studia Math. {\bf 96} (1990), 145--158.


\bibitem{SW}
E.M. Stein and G.Weiss, {\it Introduction to Fourier Analysis on
Euclidean Spaces}, Princeton Univ. Press,~1971.




\end{thebibliography}
\end{document}